\documentclass[10pt]{amsart}
\usepackage{amsmath,amssymb}
\usepackage{enumerate}
\usepackage{color}
\usepackage{hyperref}
\usepackage{graphicx}
\usepackage{dsfont}

\newtheorem{theorem}{Theorem}[section]
\newtheorem{proposition}[theorem]{Proposition}
\newtheorem{lemma}[theorem]{Lemma}

\newtheorem{definition}[theorem]{Definition}
\newtheorem{remark}[theorem]{Remark}
\newtheorem{example}[theorem]{Example}

\newtheorem{innercustomthm}{Theorem}
\newenvironment{customthm}[1]
  {\renewcommand\theinnercustomthm{#1}\innercustomthm}
  {\endinnercustomthm}

\newcommand{\R}{\mathbb{R}}

\newcommand{\N}{\mathbb{N}}

\renewcommand{\epsilon}{\varepsilon}

\begin{document}

\author{Johannes Wiesel}

\address{Johannes Wiesel\newline
Columbia University, Department of Statistics\newline
1255 Amsterdam Avenue\newline
New York, NY 10027, USA}
\email{johannes.wiesel@columbia.edu}

\thanks{
We thank Massimo Marinacci for helpful comments.}


\title{On a Lusin theorem for capacities}

\date{\today}
\begin{abstract}
Let $X$ be a compact metric space and let $v$ be a sub-additive capacity defined on $X$. We show that Lusin's theorem with respect to $v$ holds if and only if $v$ is continuous from above.
\end{abstract}

\maketitle

\section{Introduction}

Throughout this article, we work on a compact metric space $(X,d)$.  Let us quickly recall the classical Lusin theorem, which reads as follows:

\begin{theorem}[Lusin, \cite{lusin1912proprietes}]\label{thm:lusin}
Let $\mu$ be a (Borel) probability measure on $X$ and fix $\epsilon>0$. Then for every Borel-measurable function $u:X\to \R$ there exists a compact set $K=K(u, \epsilon)$ such that $\mu(X\setminus K)\le \epsilon$ and the restriction $u|_{K}$ is continuous.
\end{theorem}

In recent years, several articles have identified sufficient conditions for replacing the measure $\mu$ by a class of capacities or sublinear expectations: see e.g. \cite{cast} for an extension to Dempster capacities, \cite{congxin1994regularity, li2004lusin} for a Lusin theorem with respect to fuzzy measures, as well as \cite{zong2020egoroff} for a result concerning g-Expectations.\\

The aim of this article is to unify these recent advances by identifying a necessary and sufficient condition for an extension of Lusin's theorem to the class of sub-additive capacities. To the best of our knowledge this is the first result of its kind. Our main result can be stated as follows:

\begin{theorem}\label{thm:main}
Let $(X,d)$ be a compact metric space and let $v$ be a sub-additive capacity.
Then the property
\begin{quote}
for all $(O_n)_{n\in \N}$, $O_n$ open, $O_n\downarrow O \quad \Rightarrow \quad \lim_{n \to \infty} v(O_n)=v(O)$
\end{quote}
holds if and only if the property
\begin{quote}
for all $\epsilon>0$ and for all Borel measurable functions $u:X\to \R$ there exists a compact set $K=K(u, \epsilon)\subseteq X$, such that $v(X\setminus K)\le \epsilon$ and $u|_{K}$ is continuous
\end{quote}
is satisfied.
\end{theorem}

\begin{remark}
As any probability measure $\mu$ is a subadditive capacity, Theorem \ref{thm:main} implies Theorem \ref{thm:lusin}.
\end{remark}

In conclusion, Lusin's theorem can only hold for a class of ``very regular" capacities. This class comprises e.g. the supremum of a finite number of probability measures, but not the supremum of an infinite number of probability measures in general: indeed, classical examples of 2-alternating capacities, which are e.g. given in \cite{huber1973minimax} and are motivated by distributional robustness, do not typically satisfy the above continuity condition as we will discuss in Section \ref{sec:counterex}.\\

The remainder of the article is structured as follows: In Section \ref{sec:not} we recall basic notions in the theory of capacities and set up the notation used throughout the article. We then provide a counterexample to the main result of \cite{cast} in Section \ref{sec:counterex}, motivating the quest for stricter regularity conditions on the class of sub-additive capacities. In Section \ref{sec:main} we recall and discuss our main result, before we provide its proof in Section \ref{sec:proof}.

\section{Notation}\label{sec:not}

\begin{definition}
Let $(X,d)$ be a metric space and let $\mathcal{B}$ denote its Borel $\sigma$-algebra. 
\begin{itemize}
\item A capacity $v:\mathcal{B}\to \R$ is a set function satisfying
\begin{enumerate}
\item $v(\emptyset)=0, \quad v(X)=1$,
\item $A\subseteq B \quad \Rightarrow \quad v(A)\le v(B)$,
\item $A_n \uparrow A \quad \Rightarrow \quad \lim_{n \to \infty}v(A_n)= v(A)$,
\end{enumerate}
\item A Choquet capacity fulfils \textit{(1)}-\textit{(3)} and 
\begin{align}\label{eq:Choquet} \tag{C}
F_n \text{ closed, }F_n \downarrow F \quad \Rightarrow \quad \lim_{n \to \infty}v(F_n)= v(F).
\end{align}
\item A Dempster capacity fulfils \textit{(1)}-\textit{(3)} and \eqref{eq:Choquet} for all sequences of sets $(A_n)_{n \in \N}$ such that $A_n \supseteq \bar{A}_{n+1}$ for all $n\in \N$.
\end{itemize}
\end{definition}

\begin{definition}
A capacity $v$ is called sub-additive, if $v(A\cup B)\le v(A)+v(B)$ for all Borel sets  $A,B$.
\end{definition}

\begin{definition} 
A capacity $v$ is called $2$-alternating or submodular, if
\begin{align*}
v(A\cup B) +v(A\cap B) \le v(A)+v(B)
\end{align*}
for all Borel sets $A,B$.
\end{definition}

\begin{definition}
We define the core of a capacity $v$ as 
\begin{align*}
\text{core}(v)&:= \{ \nu \in \mathcal{P}(X) \ : \ \nu(A)\le v(A) \text{ for all }A\in \mathcal{B}\},
\end{align*}
where $\mathcal{P}(X)$ denotes the set of Borel-probability measures on $X$.
\end{definition}

For a Borel set $A$ and $\delta>0$ we define its open $\delta$-neighbourhood as

\begin{align*}
A^{\delta}:=\inf \left\{ x\in X \ : \ \inf_{y\in A} d(x,y)< \delta\right\}
\end{align*}

and its closed $\delta$-neighbourhood as 
\begin{align*}
A^{\bar{\delta}}:=\inf \left\{ x\in X \ : \ \inf_{y\in A} d(x,y)\le \delta\right\}.
\end{align*}

We also denote the closure of a set $A\in \mathcal{B}$ by $\overline{A}$.

\section{A counterexample to \cite{cast}}\label{sec:counterex}

The main result of \cite{cast} is the following extension of Lusin's theorem to $2$-alternating Dempster capacities:
\begin{proposition}\label{prop:lusin_cap}
Let $v$ be a $2$-alternating Dempster capacity defined on $\mathcal{B}$. Suppose that $u:X\to \R$ is a Borel measurable function. Given $\epsilon>0$ there exists a compact set $K$ such that $v(X\setminus K)\le \epsilon$ and $u|_{K}$ is continuous.
\end{proposition}

Let us now consider the following example:
\begin{example}[see {\cite[Example 5]{huber1973minimax}}]\label{ex:huber}
Let $\epsilon, \delta>0$, $X$ be a compact metric space and fix a probability measure $\mu$ on $X$. Then according to \cite[Example 5]{huber1973minimax}, $v$ defined by $$v(A)=\min (\mu(A^{\bar{\delta}})+\epsilon,1)$$ for all closed sets $A\neq \emptyset$  is a $2$-alternating capacity. We recall here that  $$A^{\bar{\delta}}:=\inf \left\{ x\in X \ : \ \inf_{y\in A} d(x,y)\le \delta\right\}$$ is the closed $\delta$-neighbourhood of $A$. Furthermore 
\begin{align*}
\text{core}(v)= \{\nu \in \mathcal{P}(X)\ : \ \nu(A)\le \mu(A^{\bar{\delta}})+\epsilon \text{ for all }A\in \mathcal{B}\}.
\end{align*}
We check that it is also a Dempster capacity: take any sets $A_n$ such that $A_n \supseteq \bar{A}_{n+1}$. Let $A_n\downarrow A$. As $A_n \supseteq \bar{A}_{n+1}\supseteq A_{n+1}$ this implies that $\bar{A}_{n}\downarrow A$, and as $X$ is compact, in particular $A$ is compact. Furthermore $\bar{A}_n^{\bar{\delta}}=A_n^{\bar{\delta}}$ and
thus
\begin{align}\label{eq:setequality}
 \left(\bigcap_{n \in \N} \bar{A}_n \right)^{\bar{\delta}} = \bigcap_{n \in \N} A_n^{\bar{\delta}}.
\end{align}
Indeed, the $\subseteq$-relation is clear. Now take $y\in X$ such that for all $n\in \N$ there exists some $x_n\in \bar{A}_n$ with $d(x_n, y)\le \delta$. As $\bar{A}_n$ is compact, there either exists $x\in \bigcap\bar{A}_n$ and a subsequence of $x_n$ converging to $x$, or $A_n =\emptyset$ for all large $n\in \N$. In the second case the claim is trivial. In the first case, in particular $ d(y,x)\le \limsup_{n \to \infty}d(y,x_n)\le \delta$, so $y\in  \left(\bigcap_{n \in \N} \bar{A}_n \right)^{\bar{\delta}}$. This shows \eqref{eq:setequality}. We now conclude that
\begin{align*}
\mu(A^{\bar{\delta}})=\mu\left(\left( \bigcap_{n \in \N} \bar{A}_n\right)^{\bar{\delta}}\right) =\lim_{n\to \infty} \mu (A_n^{\bar{\delta}}),
\end{align*}
which implies
\begin{align*}
v(A)=\min\left(\mu\left(A^{\bar{\delta}}\right)+\epsilon,1\right)=\lim_{n \to \infty} \min\left(\mu\left(A^{\bar{\delta}}_n\right)+\epsilon,1\right)  =\lim_{n\to \infty} v (A_n),
\end{align*}
so that $v$ is in fact a Dempster capacity.\\
On the other hand, Proposition \ref{prop:lusin_cap} does not hold: take e.g. $X=[0,1], d=|\cdot|$ and $u=\mathds{1}_{0}$, and any probability measure $\mu$. Then $u$ is clearly discontinuous. In particular for any compact set $K$ such that $u|_K$ is continuous we must have $K\neq X$, so that $v(X\setminus K) \ge \epsilon$ by definition of $v$. In particular Proposition \ref{prop:lusin_cap} is not satisfied for $\tilde{\epsilon}<\epsilon$.\\
Examing \cite{cast} the problem seems to be here: in \cite[Proof of Theorem 1]{cast} one has to check that $v$ is a capacity in the enlarged topology $\tau_u$, which makes $u$ continuous. For $u=\mathds{1}_{0}$, this means that necessarily $\tau_u$ contains $\{0\}$ and its complement. Consider now the $\tau_u$-closed sets $A_n=(0, 1/n]=(\{0\}\cup (1/n,1])^c \downarrow \emptyset$, for which $0=v(\emptyset)< \delta= v(A_n)$, so $v$ is not a capacity wrt. $\tau_u$. In particular arguments of \cite[Lemma 3]{cast} or \cite[Lemma 2.2]{huber1973minimax} do not apply and tightness of $\text{core}(v)$ is not satisfied.
\end{example}

\section{Main result}\label{sec:main}

From our discussion in Section \ref{sec:counterex} it seems that the assumptions on the capacity $v$ are not strong enough in order for a version of Lusin's theorem in the spirit of Proposition \ref{prop:lusin_cap} to hold. In fact, the continuity property from above for Dempster capacities given by 
\begin{align*}
A_n \downarrow A, \ A_n \supseteq \bar{A}_{n+1} \ \forall n \in \N \Rightarrow\quad \lim_{n \to \infty} v(A_n)=v(A)
\end{align*}
is not sufficient, as it does not demand any continuity from above for the pathological case, where the sequence of $(A_n)_{n\in \N}$ satisfies $A_n \downarrow \emptyset$ and $A_n \neq \emptyset$ for all $n\in \N$.\\
Contrary to this, continuity from above on the open sets seems to be crucial as stated in our main result, which we recall here for the the reader's convenience:

\begin{customthm}{1.1}
Let $(X,d)$ be a compact metric space and let $v$ be a sub-additive capacity. Then the property
\begin{quote}
for all $(O_n)_{n\in \N}$, $O_n$ open, $O_n\downarrow O \quad \Rightarrow \quad \lim_{n \to \infty} v(O_n)=v(O)$
\end{quote}
holds if and only if the property
\begin{quote}
for all $\epsilon>0$ and for all Borel measurable functions $u:X\to \R$ there exists a compact set $K=K(u, \epsilon)\subseteq X$, such that $v(X\setminus K)\le \epsilon$ and $u|_{K}$ is continuous
\end{quote}
holds.
\end{customthm}

\begin{remark}
Either from first principles or by checking the proof below, one concludes that the condition 
\begin{quote}
for all $(O_n)_{n\in \N}$, $O_n$ open, $O_n\downarrow O \quad \Rightarrow \quad \lim_{n \to \infty} v(O_n)=v(O)$
\end{quote}
for a capacity $v$ is actually equivalent to 
\begin{quote}
for all $(A_n)_{n\in \N}$, $A_n\downarrow A \quad \Rightarrow \quad \lim_{n \to \infty} v(A_n)=v(A)$,
\end{quote}
i.e. $v$ is continuous from above.
\end{remark}

In Example \ref{ex:huber} we have already observed that $v$ defined as the supremum of probability measures in a Prokhorov-ball around $\mu$ is not continuous from above on open sets. 
However, we can consider the following simple setting:

\begin{example}
Let $m \in \N$ and $\mu_1, \dots, \mu_m$ be probability measures on $(X,d)$. Then $v$ defined via $$v(A):=\sup_{k=1, \dots, m} \mu_k(A)$$ for all sets $A\in \mathcal{B}$ is a sub-additive capacity, which is continuous from above on the open sets.  
\end{example}
\begin{proof}
Fix a sequence of open sets $(O_n)_{n \in \N}$ such that $O_n \downarrow O$. Then as each $\mu_k$ is a proability measure, we have $\lim_{n \to \infty} \mu_k(O_n)-\mu_k(O)=0$ and thus also $$\lim_{n \to \infty} \sup_{k=1, \dots, m} \mu_k(O_n)-\mu_k(O)=\sup_{k=1, \dots, m} \lim_{n \to \infty} \mu_k(O_n)-\mu_k(O)=0.$$
The remaining properties of $v$ follow directly from the fact, that $\mu_k$ are probability measures.
\end{proof}

Lastly we remark that working on a compact metric space is not a severe restriction in this paper as the following result shows:

\begin{lemma}
Fix $\epsilon>0$. If $v$ is a Choquet capacity, then there exists a compact set $K\subseteq X$ such that $v(X\setminus K)\le \epsilon$.
\end{lemma}

\begin{proof}
This follows from \cite{huber1973minimax}[Lemma 2.2], which states that the core of $v$ is tight.
\end{proof}

\section{Proof of the main result}\label{sec:proof}

We start by showing that continuity from above on open sets implies Lusin's theorem for $v$. For this, we first introduce a stronger notion of regularity for capacities in comparison to \cite{cast}:

\begin{definition}
A set function $v:\mathcal{B}\to [0,1]$ is called regular, if the following holds: for all Borel sets $A$ and all $\epsilon>0$ there exists an open set $O$ and a closed set $F$ such that $F\subseteq A \subseteq O$ and $v(O\setminus F)\le \epsilon$. 
\end{definition}

With this definition at hand, the next proposition follows naturally:
\begin{proposition}\label{prop:reg_cap}
Let $(X,d)$ be a compact metric space and let $v$ be a sub-additive capacity satisfying
\begin{align*}
\text{for all }(O_n)_{n\in \N},\ O_n \text{ open, } O_n\downarrow O \quad \Rightarrow \quad \lim_{n \to \infty} v(O_n)=v(O).
\end{align*}
Then $v$ is regular.
\end{proposition}

\begin{proof}
We define $\mathcal{M}$ as the collection of all Borel sets, which are regular for $v$, i.e. for all $\epsilon>0$ there exists an open set $O$ and a closed set $F$ such that $F\subseteq A \subseteq O$ and $v(O\setminus F)\le \epsilon$. We show that $\mathcal{B}\subseteq \mathcal{M}$. Note that $\emptyset\in \mathcal{M}$, as it is clopen. Furthermore $\mathcal{M}$ is closed under complementation: indeed take a set Borel set $A$ and  $\epsilon>0$ such that there exists an open set $O$ and a closed set $F$ with $F\subseteq A \subseteq O$ and $v(O\setminus F)\le \epsilon$. Then we also have $O^c\subseteq A^c \subseteq F^c$, $F^c$ is open and $O^c$ is closed. Moreover, $v(F^c \setminus O^c)=v(O\setminus F)\le \epsilon$. In conclusion $A^c\in \mathcal{M}$.\\
 Next we fix a collection $\{A_n\}_{n=1}^\infty$ of sets in $\mathcal{M}$. Our aim is to show that $\bigcup_{n=1}^\infty A_n \in \mathcal{M}$. Let us fix $\epsilon>0$, and for each $n\in \N$, let us take some open sets $O_n$ and closed sets $F_n$ such that $F_n \subseteq A_n \subseteq O_n$ and $v(O_n \setminus F_n)\le 2^{-n-1}\epsilon$. The we conclude that
\begin{align*}
\bigcup_{n=1}^N F_n \subseteq \bigcup_{n=1}^\infty A_n \subseteq \bigcup_{n=1}^\infty O_n
\end{align*}
for any $N\in \N$, $\bigcup_{n=1}^N F_n$ is closed, $\bigcup_{n=1}^\infty O_n$ is open and thus by continuity from above on open sets $(4)$
\begin{align}\label{eq:open}
v \left( \bigcup_{n=1}^\infty O_n \setminus \bigcup_{n=1}^N F_n \right)\downarrow v\left( \bigcup_{n=1}^\infty O_n \setminus \bigcup_{n=1}^\infty F_n \right)
\end{align}
for $N\to \infty$. On the other hand, by monotonicity (2) and sub-additivity (5) of $v$ we conclude that
\begin{align*}
v\left( \bigcup_{n=1}^N O_n \setminus \bigcup_{n=1}^\infty F_n \right) &\le \sum_{n=1}^N v\left( O_n \setminus \bigcup_{n=1}^\infty F_n \right)\\
&\le  \sum_{n=1}^N v\left( O_n \setminus F_n \right)\\
&\le \sum_{n=1}^N 2^{-n-1}\epsilon.
\end{align*}
Lastly, $$\bigcup_{n=1}^N O_n \setminus \bigcup_{n=1}^\infty F_n\uparrow \bigcup_{n=1}^\infty O_n \setminus \bigcup_{n=1}^\infty F_n$$ for $N\to \infty$, so that the above together with continuity from below (3) implies 
\begin{align*}
v\left( \bigcup_{n=1}^\infty O_n \setminus \bigcup_{n=1}^\infty F_n \right)
&=\lim_{N\to \infty} v\left( \bigcup_{n=1}^N O_n \setminus \bigcup_{n=1}^\infty F_n \right)\\
&\le \lim_{N \to \infty} \sum_{n=1}^N 2^{-n-1}\epsilon
\le \frac{\epsilon}{2}.
\end{align*}
In particular, using \eqref{eq:open}, there exists $N\in \N$ such that 
\begin{align*}
v\left( \bigcup_{n=1}^\infty O_n \setminus \bigcup_{n=1}^N F_n \right)\le \epsilon.
\end{align*}
This shows that $\bigcup_{n \in \N} A_n\in \mathcal{M}$ and so $\mathcal{M}$ is a $\sigma$-algebra. We now show that it contains the closed sets. Indeed, take a closed set $F$ and recall that its open $\delta$-neighbourhood is given by $$F^\delta= \left\{x\in X\ : \inf_{y \in F} d(x,y)<\delta\right\}.$$ Then $F^{\delta}$ is open for $\delta>0$ and as $F$ is closed we thus have $F^\delta \setminus F \downarrow \emptyset$, which implies by continuity from above on open sets (4) that 
\begin{align*}
v(F^{\delta} \setminus F) \downarrow v(\emptyset)=0.
\end{align*}
This in particular implies that $\mathcal{B}\subseteq \mathcal{M}$, which concludes the proof.
\end{proof}

We can now state the following version of Lusin's theorem for capacities:

\begin{theorem}
Let $(X,d)$ be a compact metric space and let $v$ be a sub-additive capacity satisfying
\begin{align*}
O_n \text{ open, } O_n\downarrow O \quad \Rightarrow \quad \lim_{n \to \infty} v(O_n)=v(O).
\end{align*}

Let $u:X\to \R$ be a Borel measurable function. Given $\epsilon>0$ there exists a compact set $K$ such that $v(X\setminus K)\le \epsilon$ and $u|_{K}$ is continuous.
\end{theorem}

Given the regularity properties of $v$ established in Proposition \ref{prop:reg_cap}, this follows as in the classical case, see e.g. \cite[Proof of Theorem 7.4.3]{cohn2013measure}.

\begin{proof}
First suppose that $u$ has the representation
\begin{align*}
u=  \sum_{n=1}^\infty a_n \mathds{1}_{A_n}
\end{align*}
for real numbers $a_n$ and a countable partition of $X$ given by $\{A_n\ :\ n \in \N\}$. Fix $\epsilon>0$. By Proposition \ref{prop:reg_cap} we can find open sets $O_n$ and closed sets $F_n$ such that $F_n \subseteq A_n \subseteq O_n$ and $v(O_n \setminus F_n) \le 2^{-n-2}\epsilon$. Thus in particular
$$X=\bigcup_{n\in \N} A_n \subseteq \bigcup_{n\in \N} O_n\subseteq X,$$
so that $X=\bigcup_{n\in \N} O_n$.
As in the proof of Proposition \ref{prop:reg_cap} there exists $N\in \N$ such that 
$$v\left(\bigcup_{n=1}^\infty O_n \setminus \bigcup_{n=1}^N F_n \right)\le \epsilon/2$$ and so
\begin{align*}
v\left(\bigcup_{n=N+1}^\infty A_n\right)&=v\left(X\setminus \bigcup_{n=1}^N A_n \right) \\
&\le v\left(X\setminus \bigcup_{n=1}^N F_n\right)\\
&= v\left(\bigcup_{n=1}^\infty O_n \setminus \bigcup_{n=1}^N F_n \right)\\
&\le \epsilon/2,
\end{align*}
where we used the fact that $\{A_n \ :\  n\in \N\}$ is a partition in the first equality and $\bigcup F_n \subseteq \bigcup A_n$ for the first inequality.
Take the compact set $K=\bigcup_{n=1}^N F_n$. Again using the fact that $\{A_n \ :\  n\in \N\}$ is a partition and $F_n \subseteq A_n$, we conclude that $F_n$ are disjoint. Note that $u$ is constant on each $F_n$, so it is clearly continuous on $K$. Furthermore using sub-additivity (5) we have
\begin{align*}
v\left( X\setminus K\right) &\le v\left(\bigcup_{n=1}^N A_n \setminus \bigcup_{n=1}^N F_n\right)+v\left(\bigcup_{n=N+1}^\infty A_n\right) \\
&\le \sum_{n=1}^N v(O_n \setminus F_n)+\epsilon/2 \\
&\le \epsilon/2+\epsilon/2=\epsilon.
\end{align*}
Now let $u$ be an arbitrary Borel measurable function. Then $u$ is the uniform limit of a sequence $\left(u_{n}\right)_{n \in \N}$ of functions, each of which is Borel measurable and has only countably many values (for example, $u_{n}$ might be defined by letting $u_{n}(x)$ be $k / n,$ where $k$ is the integer that satisfies $k / n \leq u(x)<(k+1) / n$). The above then implies that for each $n\in \N$ there is a compact subset $K_{n}$ of $X$ such that $v\left(X\setminus K_{n}\right)<2^{-n}\epsilon $ and such that the restriction of $u_{n}$ to $K_{n}$ is continuous. Let $K=\bigcap_{n=1}^\infty K_{n} .$ Then $K$ is a compact subset of $X$,
$$
v(X\setminus K) \leq \sum_{n=1}^\infty \mu\left(X\setminus K_{n}\right)\le \sum_{n=1}^\infty 2^{-n}\epsilon=\epsilon
$$
and $u$ is continuous on $K$ as the uniform limit of the functions $u_{n},$ which are continuous on $K.$ This concludes the proof.
\end{proof}

We now complete the proof via the following proposition:

\begin{proposition}
Assume that $(X,d)$ is compact metric space and $v$ is a sub-additive capacity satisfying the property
\begin{quote}
for all $\epsilon>0$ and for all Borel measurable functions $u:X\to \R$ there exists a compact set $K=K(u, \epsilon)\subseteq X$, such that $v(X\setminus K)\le \epsilon$ and $u|_{K}$ is continuous.
\end{quote}
Then for all $(O_n)_{n\in \N}$, $O_n$ open, $O_n\downarrow O \quad \Rightarrow \quad \lim_{n \to \infty} v(O_n)=v(O)$.
\end{proposition}

\begin{proof}
Fix a sequence $(O_n)_{n\in \N}$ of open sets such that  $O_n\downarrow O$. Take $\epsilon>0$. Our aim is to show that $v(O_n)-v(O)\le \epsilon$ for $n$ large enough. Note that sub-additivity of $v$ implies that for Borel sets $A\subseteq B$ we have
\begin{align}\label{eq:subad}
 v(B)-v(A)\le v(B\setminus A).
\end{align}
In particular we can assume that $v(X\setminus O)>0$: indeed, if this is not the case then \eqref{eq:subad} implies that 
$$0 \le v(X)-v(O)\le v(X\setminus O)=0,$$ so that $v(O)=v(X)=1$. Thus we conclude $v(O_n)=1$ for all $n\in \N$ by monotonicity of $v$ and  $O_n \supseteq O$, which means $\inf_{n\in \N} v(O_n)-v(O)=0$, in particular  $v(O_n)-v(O)\le \epsilon$ for $n$ large enough.\\
To simplify notation we now define $$\tilde{v}(A):=\frac{v(A\setminus O)}{v(X\setminus O)}$$
for Borel sets $A\subseteq X$. We note that $\tilde{v}(A)=\tilde{v}(A\setminus O)$ and that $\tilde{v}$ is again a sub-additive capacity. Furthermore, sub-additivity of $v$ implies that 
$$\inf_{n\in \N}\frac{v(O_n)- v(O)}{v(X\setminus O)}\le  \inf_{n\in \N}\frac{v(O_n \setminus O)}{v(X\setminus O)}=\inf_{n\in \N} \tilde{v}(O_n\setminus O),$$
so it is sufficient to show that $$\tilde{v}(O_n\setminus O)\le \frac{\epsilon}{v(X\setminus O)}=:\tilde{\epsilon}$$ for $n$ large enough.
We now define the sequence of functions $(u_n)_{n \in \N}$ via
\begin{align*}
u_n:= \begin{cases}
1 & \text{ if } x\in O_n\setminus O\\
0& \text{ otherwise}.
\end{cases}
\end{align*}
Then for each $n\in \N$ we can apply Lusin's theorem to $\tilde{v}$ in order to obtain a compact set $K_n\subseteq X$ such that $\tilde{v}(X\setminus K_n )\le 2^{-n}\tilde{\epsilon}$ and $u_n|_{K_n}$ is continuous. Then necessarily 
\begin{align}\label{eq:abstand}
\inf\left\{ d(x,y) \ :\ x\in K_n\cap (O_n \setminus O), y \in K_n \cap (O_n \setminus O)^c\right\}>0.
\end{align}
Indeed otherwise there exists $x\in K_n$ such that $x\in \overline{K_n\cap (O_n \setminus O)} \cap \overline{K_n\cap (O_n \setminus O)^c}$, which contradicts continuity of $u_n|_{K_n}$. From \eqref{eq:abstand} we then conclude that $K_n\cap (O_n \setminus O)$ is compact. Now we define a new sequence $(\tilde{K}_n)_{n \in \N}$ of compact sets given by $$\tilde{K}_n =\bigcap_{l=1}^n K_l\subseteq K_n$$ for all $n\in \N$. Then clearly $\tilde{K}_n$ is decreasing, $\tilde{K}_n \cap (O_n\setminus O)$ is compact as a finite intersection of compact sets and 
\begin{align}\label{eq:long}
\begin{split}
\tilde{v}((O_n \setminus O)\setminus \tilde{K}_n) &\le \tilde{v}(X\setminus \tilde{K}_n) \\
&= \tilde{v}\left( X\setminus \bigcap_{l=1}^n K_l \right)\\
&=\tilde{v}\left( \bigcup_{l=1}^n X\setminus K_l \right)\\
&\le \sum_{l=1}^n \tilde{v}(X\setminus K_l) \\
&\le \sum_{l=1}^n  2^{-l}\tilde{\epsilon}\\
&\le \tilde{\epsilon}
\end{split}
\end{align}
for all $n \in \N$. Furthermore
\begin{align*}
\tilde{K}_n \cap (O_n\setminus O) \subseteq O_n \setminus O \downarrow \emptyset
\end{align*}
as $n \to \infty$, which directly implies
\begin{align*}
\tilde{K}_n \cap (O_n\setminus O) \downarrow \emptyset
\end{align*}
as $n \to \infty$. As $\tilde{K}_n \cap (O_n\setminus O)$ is compact and decreasing, there exists $n_0\in \N$ such that
\begin{align*}
\tilde{K}_n \cap (O_n\setminus O) =\emptyset
\end{align*}
for all $n\ge n_0$. In particular
\begin{align}\label{eq:limit}
\tilde{v} (\tilde{K}_n \cap (O_n\setminus O))=\tilde{v}(\emptyset)=0
\end{align}
for all $n\ge n_0$. In conclusion, using sub-additivity of $\tilde{v}$, \eqref{eq:long} and \eqref{eq:limit},
\begin{align*}
\tilde{v} (O_n \setminus O) &\le \tilde{v}(\tilde{K}_n \cap (O_n \setminus O)) + \tilde{v}((O_n \setminus O) \setminus \tilde{K}_n )\\
&\le \tilde{\epsilon}
\end{align*}
for $n\ge n_0$. This concludes the proof.
\end{proof}

\bibliographystyle{plain}
\bibliography{bib}
\end{document}